\documentclass[11pt]{amsart}
\usepackage{amsmath}
\usepackage{amsfonts}
\usepackage{amssymb}
\newtheorem{theorem}{Theorem}[section]
\newtheorem{corollary}[theorem]{Corollary}

\newtheorem{lemma}[theorem]{Lemma}
\newtheorem{proposition}[theorem]{Proposition}

\theoremstyle{definition}

\newtheorem{definition}[theorem]{Definition}
\newtheorem{remark}[theorem]{Remark}
\newtheorem{example}[theorem]{Example}

\newcommand{\C}{\mathbb{C}}
\newcommand{\PF}{\mathcal{P}\mathcal{F}(n)}
\newcommand{\R}{\mathbb{R}}

\newcommand{\ta}{\tau_{\mathcal{A}}}
\newcommand{\tb}{\tau_{\mathcal{B}}}

\begin{document}

\title{Positively Factorizable Maps }

\author[Jeremy Levick]{Jeremy Levick}
\address{Institute for Quantum Computing and Department of Mathematics, University of Guelph,
Guelph, ON, Canada}
\email{levickje@uoguelph.ca }

\author[Mizanur Rahaman]{Mizanur Rahaman}
\address{BITS Pilani KK Birla Goa Campus, Goa India}
\email{mizanurr@goa.bits-pilani.ac.in}
\begin{abstract}
We initiate a study of linear maps on $M_n(\C)$ that have the property that they factor through a tracial von Neumann algebra $(\mathcal{A,\tau})$ via operators $Z\in M_n(\mathcal{A})$ whose entries consist of positive elements from the von-Neumann algebra. These maps often arise in the context of non-local games especially in the synchronous case. We establish a connection with the convex sets in $\mathbb{R}^n$ containing self-dual cones and the existence of these maps. The Choi matrix of a map of this kind which factor through an abelian von-Neumann algebra turns out to be a completely positive (CP) matrix.   
 We fully characterize positively factorizable  maps whose Choi rank is 2. We also provide some applications of this analysis in finding doubly nonnegative matrices which are not CPSD. A special class of these  examples are found from the concept of Unextendible Product Bases in quantum information theory.
\end{abstract}

\maketitle

%%%%%%%%%%%%%%%%%%%%%%
\section{Introduction} 
We study trace preserving completely positive maps (quantum channels) $\Phi:M_n(\C)\rightarrow M_n(\C)$ such that there exists a finite von Neumann algebra $\mathcal{A}$ with a normal faithful trace $\tau$ and an operator $Z\in M_n(\mathcal{A})$ whose entries are all positive elements of $\mathcal{A}$ and $\Phi$ is given by \begin{equation}\label{equ-def-intro}
\Phi(X)=id\otimes \tau(Z(X\otimes 1_{\mathcal{A}})Z^*), \ \forall X\in M_n(\C).
\end{equation}
We call such maps \textbf{positively factorizable} and denote this set as $\PF$.

One of the motivations for studying these maps is the frequent occurrence of these maps in the context of non-local games, especially synchronous non-local games (\cite{psstw},\cite{hmkps},\cite{ortiz-paulsen}). Recall that a probability density $p(a,b|x,y)$ in a synchronous game between two players with input set $I$ of cardinality $n$ and output set $O$ of cardinality $k$ is given by a tracial C$^*$-algebra $(\mathcal{A},\tau)$ generated by projections $\{Q_{x,a}\}_{x,a}$ satisfying $\sum_{a=1}^k Q_{x,a}=1$, $\forall x$, such that 
\[p(a,b|x,y)=\tau(Q_{x,a}Q_{y,b}).\]
Given such a probability density $p(a,b|x,y)$, one defines a map $\Phi_{p}:M_n(\C)\rightarrow M_k(\C)$ by
\[\Phi_p(E_{x,y})=\sum_{a,b}p(a,b|x,y)E_{a,b},\]
where $\{E_{i,j}\}$ denote the matrix units and extend $\Phi_p$ by linearity. Then it was proven in \cite{ortiz-paulsen} that if $p(a,b|x,y)$ is a synchronous density, then $\Phi_p$ is a completely positive map. 
Maps of this form were further analyzed in a more restrictive class of games, called bisynchronous games in \cite{PR}. It can be seen that these maps are all examples of positively factorizable maps when $n=k$, where the operator $Z$ given in the Equation \ref{equ-def-intro} is obtained by
\[Z=(Q_{x,a})\in M_n(\mathcal{A}).\] 
%%%%%%%%%%%
Maps on $M_n(\C)$ which factor through a tracial von-Neumann algebra have been introduced by Anantharaman-Delaroche (\cite{A-D}) and analyzed by Haagerup and Musat (\cite{H-M1},\cite{H-M2}). In their notion of factorizability the Equation \ref{equ-def-intro} holds with $Z$ being a unitary operator in $M_n(\mathcal{A})$. 
The notion of factorizability we are concerned with is intrinsically different from the factorizable maps studied by Haagerup and Musat. It turns out that the Choi matrix of a  positively factorizable map lies in the closure of the set of completely positive semidefinite(CPSD) matrices (see \cite{Laurent1}, \cite{roberson})-a cone which has been recently introduced to exhibit linear conic formulations of various quantum graph parameters. In \cite{relaxation} the authors studied linear maps whose Choi matrices lie in various cones of symmetric matrices. Our investigation here is closer to the spirit of this approach. However, in \cite{relaxation} the analysis has been carried out keeping the graph isomorphism game (\cite{graph-iso}) in the background. In our treatment we emphasize more the general theory of completely positive maps and convex cones. 

%%%%%%%%%%%%
\section{Notation and Basic Definitions}

%%%%%%%%%%%5
We will often use the symbol $\succeq$ to represent an element of a C$^*$-algebra or a matrix to be positive semidefinite, while we use $\geq 0$ to represent entrywise non-negativity. So an $A\succeq 0$ means the element $A$ is positive semidefinite (or simply positive) and $A\geq 0$, means the element $A$ is entrywise nonnegative.  

In all that follows, $\Phi$ is a quantum channel, a trace-preserving completely positive map on $M_n(\C)$, with Kraus operators $\{K_i\}_{i=1}^d\subseteq M_n(\C)$.
That is, one can represent $\Phi$ as follows 
\[\Phi(X)=\sum_{i=1}^d K_i X K_i^*, \ \forall X\in M_n(\C).\]
The trace-preserving condition is equivalent to having $\sum_{i=1}^d K_i^*K_i=1$. The Choi-rank of $\Phi$ is the minimal number of linearly independent $K_i$'s required to represent $\Phi$ as above.
 The Choi matrix of $\Phi$ is $$J(\Phi) = \sum_{i,j=1}^n E_{ij} \otimes\Phi(E_{ij}).$$
 
\begin{definition} A channel is \emph{factorizable} (see \cite{H-M1}) if either of the following equivalent conditions hold:
\begin{enumerate}
\item There exists a von Neumann algebra $\mathcal{A}$ with faithful trace $\tau$ such that 
$$\Phi(X) = (\mathrm{id}\otimes \tau)\biggl( U (X\otimes I_{\mathcal{A}})U^*\biggr)$$ where $U$ is a unitary in $M_n(\C) \otimes \mathcal{A}$
\item There exists a von Neumann algebra $\mathcal{A}$ with trace $\tau$ and elements $\{A_i\}_{i=1}^d$ satisfying 
$$\tau(A_i^*A_j) = \delta_{ij}$$ such that 

$$U := \sum_{i=1}^d K_i \otimes A_i$$ is unitary.
\end{enumerate}
\end{definition}

The equivalence of the two conditions is seen as follows: 

If $(2)$ holds, then 

\begin{align*} (\mathrm{id}\otimes \tau)\biggl(U(X\otimes I_{\mathcal{A}})U^*\biggr) & = \sum_{i,j=1}^d K_i X K_j^* \tau (A_iA_j^*) \\
& = \sum_{i=1}^d K_i X K_i^* \\
& = \Phi(X)\end{align*}

If $(1)$ holds, we assume with no loss of generality that $\{K_i\}_{i=1}^d$ are linearly independent, so we can complete them to a basis $\{K_i\}_{i=1}^{n^2}$. Then, $U$ can be expressed in terms of this basis as 
$$U = \sum_{i=1}^{n^2} K_i \otimes A_i$$ for some $A_i \in \mathcal{A}$. 
Then $(1)$ implies that 

$$\sum_{i=1}^d K_i X K_i^* = \sum_{i=1}^{n^2} K_iXK_j^* \tau (A_i^*A_j);$$

Let $L_i$, $R_j$ be left-multiplication by $K_i$ and right-multiplication by $K_j^*$ respectively; we can represent $L_i$, $R_j$ acting on $\C^{n^2}$ by $I\otimes K_i$ and $\overline{K_j}\otimes I$ respectively, so we have that 

$$\sum_{i=1}^d \overline{K_i}\otimes K_i = \sum_{i=1}^{n^2} \tau(A_i^*A_j)\overline{K_j}\otimes K_i$$ as these two operators act the same on each $x \in \C^{n^2}$. However, the set $\{\overline{K_j}\otimes K_i\}$ is linearly independent, since without loss of generality we can pick $\{K_i\}_{i=1}^{n^2}$ to be mutually orthogonal, so 

$$\mathrm{tr}\bigl[ (\overline{K_j}\otimes K_i)^*(\overline{K_k}\otimes K_l)\bigr] = \mathrm{tr}(K_j^T\overline{K_k})\mathrm{tr}(K_i^*K_l) = \delta_{jk}\delta_{il}.$$

So we must have that $\tau(A_i^*A_j) = \delta_{ij}$ for $i,j\leq d$ and $0$ otherwise. Since for all $i>d$, $\tau(A_i^*A_i) = 0$, $A_i = 0$ and so 
$$U = \sum_{i=1}^d K_i \otimes A_i.$$

By analogy, we define the following:
\begin{definition}\label{def:PF}
 A channel $\Phi$ with $d$ linearly independent Kraus operators $\{K_1,\cdots,K_d\}$, is \textbf{positively factorizable ($\PF$)} if it satisfies either of the following equivalent conditions:
\begin{enumerate}
\item There exists a von Neumann algebra $\mathcal{A}$ with faithful trace $\tau$ and a matrix $Z \in M_n(\C) \otimes \mathcal{A}$ such that 
$$\Phi(X) = (\mathrm{id}\otimes \tau)\bigl( Z (X \otimes I_{\mathcal{A}})Z^*\bigr)$$ where 
the $(i,j)$ block of $Z$, $Z(i,j)$, is a positive element in $\mathcal{A}$, for all $(i,j)$.
\item There exists a von Neumann algebra $\mathcal{A}$ with faithful trace $\tau$ and elements $\{A_i\}_{i=1}^d$ satisfying $\tau(A_i^*A_j)=\delta_{ij}$ such that the matrix
$$Z = \sum_{i=1}^d K_i \otimes A_i$$ has its $(i,j)$ block a positive element of $\mathcal{A}$ for all $i,j \leq n$. 
\end{enumerate}
\end{definition}

The equivalence of these two conditions is essentially the same proof as for regular factorizability: 
$(2)$ implies $(1)$ is simply the result of the computation 

$$(\mathrm{id}\otimes \tau)\bigl(Z(X \otimes I_{\mathcal{A}})Z^*\bigr) = \sum_{i,j=1}^d K_i X K_j^* \tau(A_i^*A_j) = \Phi(X)$$

and $(1)$ implies $(2)$ again involves writing $Z = \sum_{i=1}^{n^2}K_i \otimes A_i$ where we complete $\{K_i\}_{i=1}^d$ to a full basis for $M_n(\C)$ and then use linear independence to show that $\tau(A_i^*A_j) = \delta_{ij}$ for $i,j\leq d$ and $0$ otherwise. 
\begin{definition} Let $A = (A_1,\cdots, A_d) \subseteq \mathcal{A}^d$ be a $d$-tuple of operators in $\mathcal{A}$; the \emph{joint numerical range} is the subset of $\C^d$ given by 

$$W(A) = \{ (x^*A_1 x, \cdots, x^*A_dx) : \|x\|=1\};$$

this is written as if $A \subseteq \mathcal{B}(\mathcal{H})$ for some Hilbert space $\mathcal{H}$, in which case $x$ is a vector from the unit ball of $\mathcal{H}$. 
\end{definition}
Recall that a cone $C$ in a vector space $V$ is a subset such that $\alpha x$ is in $V$ for any $x\in C$ and any positive real $\alpha$. 
\begin{definition} The \emph{dual cone} of any subset $S$ of $\C^d$ is the set 
$$S^* := \{ y: \langle y,x\rangle \geq 0 \ \forall \ x \in S\}.$$
\end{definition}

$S^*$ is always a closed convex cone, no matter what kind of set $S$ is: this is because all positive linear combinations of vectors in $S^*$ remain in $S^*$ since 

$$\langle \sum_i p_i y_i,x\rangle = \sum_i p_i \langle y_i,x\rangle$$ so if $\langle y_i,x\rangle \geq 0$ for all $x \in S$, the LHS is positive too for any $p_i \geq 0$. 

Closure follows because of the continuity of the inner product. 

\begin{proposition} \label{Prop:W(A)^*=D(A)}
Let $A \subseteq \mathcal{A}^d$ be a $d$-tuple of Hermitian operators; then $W(A)^*$ is the set of coefficients of a linear combination $y = (y_1,\cdots, y_d)$ such that 
$$\sum_{i=1}^d y_i A_i \succeq 0.$$
\end{proposition}

\begin{proof} This follows from the observation that if $W(A) \ni z = (x^*A_1x, \cdots, x^*A_dx)$, then 

$$\langle y,z\rangle = x^*\bigl(\sum_{i=1}^d y_i A_i\bigr)x$$ so $\langle y,z\rangle \geq 0$ for all such $z$ is equivalent to 
$$x^*\bigl(\sum_{i=1}^d y_i A_i \bigr)x \geq 0$$ for all $x$, which is equivalent to the $\sum_{i=1}^d y_i A_i \succeq 0$.
\end{proof}
\begin{definition} For a tuple of Hermitian operators $A = (A_1,\cdots, A_d)$, define the set \[D(A)=\{y=(y_1,\cdots,y_d) \in \R^d : \sum_{i=1}^d y_iA_i \succeq 0\},\] where $\succeq 0$ indicates a positive semidefinite operator. 
\end{definition}
It follows that this set is closed convex cone. This set is known as the classical spectrahedron associated with the operators $A_1,\cdots,A_d$ (see \cite{spectrahedra} and references therein).

Before we further explore these concepts in our setting, we need a lemma.
\begin{lemma}\label{lemma:real kraus}
If $\Phi$ is a map on $M_n(\C)$ with Kraus operators $\{K_1,\cdots, K_d\}$ and belongs to the set $\PF$ by means of $Z=\sum_{i=1}^d K_i\otimes A_i$, then one can choose another set of Kraus operators $\{K_i'\}$ representing $\Phi$ such that  $K_i'\in M_n(\mathbb{R})$, for all $i$.
\end{lemma}
\begin{proof}
We first observe that the Choi matrix is a positive semidefinite matrix which is entrywise real. To see this note that if $\Phi$ is in $\PF$, then 
there exists a finite von-Neumann algebra $(\mathcal{A},\tau)$ and an operators $Z=(Z({i,j}))\in M_n\otimes \mathcal{A}$ with 
$Z({i,j})$ positive  in $\mathcal{A}$ for all $(i,j)$ such that 
\[\Phi(x)=id\otimes \tau(Z(x\otimes 1)Z^*).\]
Then the Choi matrix
\[J(\Phi)=id\otimes \Phi(\sum_{i,j}E_{i,j}\otimes E_{i,j})=\sum_{i,j}E_{i,j}\otimes \sum_{k,l}\tau(Z({k,i})Z({l,j}))E_{k,l}.\]
Clearly the Choi matrix is entry wise real as $\tau(Z({k,i})Z({l,j}))\geq 0$. As the Kraus operators arise from the eigenvectors of the Choi matrix, the assertion follows from the fact that any positive semidefinite real matrix has a basis of real eigenvectors. 

%To see this, if $Av=\lambda v$ for a real psd matrix $A$ with eigenvalue $\lambda$ and eigenvector $v$, then one has $A(Re(v))=\lambda (Re(v))$ and $A (Im(v))=\lambda (Im(v))$, where $Re(v)$ and $Im(v)$ are the real and imaginary parts of the vector $v$. So here are two eigenvectors corresponding to $\lambda$ and one of them has to be non-zero.   
\end{proof}

\begin{proposition}\label{makespos}$\Phi$ is in $\PF$ by means of the matrix $Z = \sum_{i=1}^d K_i \otimes A_i$ for $A = (A_1,\cdots, A_d) \subseteq \mathcal{A}$ if and only if 
$$\sum_{i=1}^d y_i K_i \geq 0$$ for all $y=(y_1,\cdots,y_d) \in W(A)$, where $\geq 0$ here means entrywise nonnegative.
\end{proposition}
\begin{proof} Suppose $Z = \sum_{i=1}^d K_i \otimes A_i$ has the property that each block $Z(i,j)$ is positive. By expanding $K_i = \sum_{p,q=1}^n k^{(i)}_{p,q} E_{pq}$ we see that 
$$Z(p,q) = \sum_{i=1}^d k^{(i)}_{p,q} A_i$$ and so if this is positive, then 
$$0 \leq x^*Z(p,q)x = \sum_{i=1}^d k^{(i)}_{p,q} x^*A_i x = \sum_{i=1}^d y_i k^{(i)}_{pq}$$ for all $y = (y_1,\cdots, y_d) = (x^*A_1x, \cdots, x^*A_d x) \in W(A)$. 

Thus, $$\sum_{i=1}^d y_i K_i$$ has as its $(p,q)$ entry $\sum_{i=1}^d y_k k^{(i)}_{pq} \geq 0$. 

The converse follows from reversing the steps.
\end{proof}

\begin{definition} For a tuple of matrices $K = (K_1,\cdots, K_d) \in M_n(\mathbb{R})^d$, define the \emph{nonnegativity cone}, $$NC(K) : = \{ v \in \C^d: \sum_{i=1}^d v_i K_i \geq 0\}.$$
Again, we use $\geq 0$ to mean entrywise nonnegativity. 
\end{definition}
Note that by the Lemma \ref{lemma:real kraus}  we know the Kraus operators can be chosen to be real and also we can choose them to be linearly independent. Now if $v=(v_1,\cdots,v_d)\in \mathbb{C}^d$ such that $\sum_{i}v_iK_i\geq 0$, write $v=(v_1,\cdots,v_d)=(x_1+iy_1,\cdots,x_d+iy_d)$ splitting each component into real and imaginary parts. Then we have 
\[\sum_{j=1}^d x_jK_j+i\sum_{j=1}^dy_jK_j\geq 0.\]
As the second term is entirely imaginary, we must have $\sum_{j=1}^dy_jK_j=0$. As $\{K_i\}$ are linearly independent, we must have $y_j=0$, for all $j$, so $v$ is actually a real vector. Moreover, since we want to find $A=(A_1,\cdots,A_d)$ such that $W(A)=\{(v^*A_1v,\cdots,v^*A_dv)\}\subseteq NC(K)$, we must require $W(A)\subseteq \mathbb{R}^d$. This means $v^*A_iv\in \mathbb{R}$ for each $i$, and so all $A_i$ must be Hermitian. So from now on, we will define 
$$NC(K) : = \{ v=(v_1,\cdots,v_d) \in \R^d: \sum_{i=1}^d v_i K_i \geq 0\}.$$

We also use the notation $K(v) := \sum_{i=1}^d v_i K_i$ for a vector $v \in \R^d$. 

That $NC(K)$ is a cone follows from the fact that if $v_i \in NC(K)$, for any positive coefficients $p_i$, if $v = \sum_i p_i v_i$, then 
\begin{align*} \sum_{i=1}^d v_i K_i & = \sum_{i=1}^d \sum_j p_j v_{ji} K_i \\
& = \sum_j p_j \sum_i v_{ji}K_i \\
& = \sum_j p_j K(v_j)\end{align*} which is a positive combination of entrywise positive matrices, and so must be positive.
%Note that by the Lemma \ref{lemma:real kraus}  we know the Kraus operators can be chosen to be real and also we can choose them to be linearly independent. Now if $v=(v_1,\cdots,v_d)\in \mathbb{C}^d$ such that $\sum_{i}v_iK_i\geq 0$. Write $v=(v_1,\cdots,v_d)=(x_1+iy_1,\cdots,x_d+iy_d)$ splitting each component into real and imaginary parts. Then we have 
%\[\sum_{j=1}^d x_jK_j+i\sum_{j=1}^dy_jK_j\geq 0.\]
%As the second term is entirely imaginary, we must have $\sum_{j=1}^dy_jK_j=0$. As $\{K_i\}$'s are linearly independent, we must have $y_j=0$, for all $j$, so $v$ is actually a real vector.

%So from now on, we will define 
%$$NC(K) : = \{ v=(v_1,\cdots,v_d) \in \R^d: \sum_{i=1}^d v_i K_i \geq 0\}.$$ 

%That $NC(K)$ is a cone follows from the fact that if $v_i \in NC(K)$, for any positive coefficients $p_i$, if $v = \sum_i p_i v_i$, then 
%\begin{align*} \sum_{i=1}^d v_i K_i & = \sum_{i=1}^d \sum_j p_j v_{ji} K_i \\
%& = \sum_j p_j \sum_i v_{ji}K_i \\
%& = \sum_j p_j K(v_j)\end{align*} which is a positive combination of entrywise positive matrices, and so must be positive. 
We summarize these observations in the following corollary.
 
\begin{corollary}\label{NCinW} Let $\Phi$ be a quantum channel with Kraus operators $\{K_i\}_{i=1}^d$, or $K = (K_1,\cdots, K_d)$. Then $\Phi$ is in $\PF$ if and only if there exists a von Neumann algebra $\mathcal{A}$ with faithful trace $\tau$ and a tuple of operators $A = (A_1,\cdots, A_d)\in \mathcal{A}^d$ satisfying $\tau(A_i^*A_j) = \delta_{ij}$ such that 

$$NC(K)^* \subseteq D(A).$$
\end{corollary}
\begin{proof} This is essentially just a restatement of previous results with the inclusion reversed due to the duality. For example, if the inclusion holds, then there exists a tuple $A = (A_1,\cdots, A_d)$ such that for all $x \in \mathcal{B}(\mathcal{H})$, $K((x^*A_1,\cdots, x^*A_d))\geq 0$, i.e.
$$\sum_{i=1}^d x^*A_ix K_i \geq 0$$ and so by Theorem \ref{makespos}, $\Phi$ is in $\PF$. 

Now, as noted earlier, for any subset $S\subseteq \C^d$ the dual cone is always a closed convex cone, regardless of whether $S$ is; if $S$ is itself a closed convex cone then $(S^*)^* = S$. Taking the dual is inclusion reversing:

$$S \subset K \Rightarrow K^* \subseteq S^*.$$

Thus, the condition that $NC(K) \supseteq W(A)$ for some trace-orthonormal $A$ may be expressed alternatively by using the Proposition \ref{Prop:W(A)^*=D(A)} as 
$$ NC(K)^* \subseteq D(A)$$ for some trace-orthonormal $A$.
\end{proof}
\begin{remark}
Note that although the joint numerical range $W(A)$ depends on the representation of the tuple $A=(A_1.\cdots, A_d)$ onto some Hilbert space, the set $D(A)$ is independent of any representation!
\end{remark}
%\textbf{Vern's Question:} Does the relation $W(A)\subseteq NC(K)$ hold for some representation of A as operators on a Hilbert space?! Or is it the case that if this is true for some representation then it is true for all representations?
%\vspace{.5cm}

Recall that a self-dual cone $C\in \R^d$ is one with $C^*=C$. Examples of self-dual cones are the nonnegative orthant which consists of of all $x\in \mathbb{R}^n$ with nonnegative components, the $n$-dimensional ice cream cone: \[K_n=\{x\in \mathbb{R}^n: {(x_1^2+\cdots+x_{n-1}^2)}^{\frac{1}{2}}\leq x_n\},\] the cone of positive semidefinite matrices in the real space of all Hermitian matrices (see \cite{barker1976self} for more examples).  
For Euclidian cones, there is an interesting fact concerning self dual cones: 
\begin{theorem}[Barker-Foran, see \cite{barker1976self}]\label{barker}
If $C\subset \R^d$ is a cone such that $C\subset C^*$, then there is a self-dual cone $K$ such that $C\subset K=K^*\subset C^*$.
\end{theorem}
We are ready for our main theorem of the section.
\begin{theorem} \label{thm-main}
Let $\Phi$ be a channel with Kraus operators $K = (K_1,\cdots, K_d)$. 
A necessary condition for $\Phi$ to be in $\PF$ is for $NC(K)$ to contain a self-dual cone within. 
\end{theorem}
\begin{proof}
We have already seen that $\Phi$ is  in $\PF$ if and only if \[NC(K)^* \subseteq D(A)\] for some tuple $A = (A_1,\cdots, A_d)$ where the $A_i$ are trace-orthonormal. 

Now we will show that that $D(A) \subseteq D(A)^*$. Then taking the dual again, we will invoke Theorem \ref{barker}.

To this end, Suppose $y \in D(A)$; that is $A(y) \succeq 0$. We must have that $\tau(A(y)^*P) \geq 0$ for all $P \succeq 0$; in particular, for all $x \in D(A)$, $A(x) \succeq 0$ and so \[\tau(A(y)^*A(x)) \geq 0.\] 
Now it follows that
\begin{align*}  \langle x,y\rangle & =
\sum_{i=1}^d \overline{x_i}y_i\\
&= \sum_{i,j=1}^d \overline{x_i}y_j \tau(A_i^*A_j) \\
&=\tau(A(x)^*A(y))\geq 0.
\end{align*} 
 We thus have that, if $y \in D(A)$, 
$$\langle y,x\rangle \geq 0$$ for all $x \in D(A)$. Hence $y \in D(A)^*$.

Hence, we have $$NC(K)^* \subseteq D(A)\subseteq D(A)^*\subseteq NC(K)$$ where the last inclusion follows from taking the dual again, and using the fact that $NC(K)$ is a closed convex cone. The assertion of the  theorem now follows from invoking Theorem \ref{barker}.
\end{proof}

\section{Factorizable via abelian ancilla}
In this section we characterize PF maps that factor through an abelian algebra. 
\begin{theorem}\label{Thm:abelian ancilla}
For a channel $\Phi$ with Kraus operators $K=(K_1,\cdots, K_d)$ the following statements are equivalent
\begin{enumerate}
\item $\Phi$ is $\PF$ via an abelian algebra. 
\item There are  vectors $\{v_i\}_{i=1}^m \subseteq NC(K)$ such that the vectors satisfy 

$$\sum_{s=1}^m p_s v_sv_s^* = \frac{1}{d}I_d$$ for some probability vector $p = (p_1,\cdots, p_m)$.
\end{enumerate}
\end{theorem}
\begin{proof}
\textbf{1 $\Longrightarrow$ 2.} Suppose $\Phi$ is $\PF$ with $\mathcal{A}$ an abelian algebra; i.e., $A_i$ are mutually diagonalizable. Suppose $A_i = \mathrm{diag}(a_i)$, and $\tau(E_{ii}) = p_i$ where $(p_1,\cdots, p_m)$ is a probability vector. Define the inner product $\langle \cdot, \cdot \rangle_p$ by 
$$\langle v,w\rangle_p := \sum_{i=1}^m p_i\overline{v_i}w_i.$$

It is easily seen that $W(A) = \{(\langle q,a_1\rangle_p, \cdots, \langle q, a_d\rangle_p) : q_i \geq 0 \ \& \ \langle q,q\rangle_p = 1\}$ since $x^*A_ix = \sum_{j=1}^m p_j |x_j|^2 a_{ij}$; if $q = (|x_1|^2, \cdots, |x_m|^2)$ then this is $\langle q,a_i\rangle_p$. 

Thus, it must be that 

\begin{align*}\sum_{i=1}^d \langle q,a_i\rangle_p K_i &= \sum_{i=1}^d \sum_{j=1}^m p_j q_j a_{ij}K_i \\
& = \sum_{j=1}^m p_j q_j \sum_{i=1}^d a_{ij}K_i\\
& \geq 0\end{align*}

for some $p_j >0$ and all $q_j \geq 0$. In particular, we can choose $ x= e_s$, normalized so that $\langle x,x\rangle_p = 1$; then $q_j = 0$ for all $j$ except $s$, and we get 

$$\sum_{i=1}^d a_{ij} K_i \geq 0$$ for each $j$. If $v_j = \sum_{i=1}^m a_{ij}e_i$, then $K(v_j) \geq 0$ for each $j$, i.e., $v_j \in NC(K)$. 

Finally, $\tau(A_i^*A_j) = \delta_{ij} = \langle a_i,a_j\rangle_p$; this is 
$$d\sum_{s=1}^m p_s \overline{a_{is}}a_{js} = d\sum_{s=1}^m p_s \overline{v_{si}}v_{sj}$$ which is the $(i,j)$ entry of $d\sum_{i=1}^m p_s v_sv_s^*$; so we have that 
$$d\sum_{s=1}^m p_s v_sv_s^* = I_d.$$

\textbf{2 $\Longrightarrow$ 1.}
Suppose now that $NC(K)$ contains vectors $\{v_i\}\subseteq \R^d$ such that $\sum_i p_iv_iv_i^*=\frac{1}{d}I_d$ for some probability vector $(p_1,\cdots,p_m)$. Then let 
$A_i=\sum_{j=1}^m v_{ji}E_{jj}$- a $m\times m$ diagonal matrix with $i^{th}$ entries of each $v_j$ as its entries. Let $\mathcal{A}$ be the (abelian) von-Neumanna algebra generated by $A_i$ with trace $\tau(E_{jj})=p_j$. Then
\[\tau(A_i^*A_j)=\sum_k p_k\overline{v_{ki}}v_{kj}=(\sum_k p_kv_kv_k^*)_{ij}.\] 
Clearly from the condition $\sum_i p_iv_iv_i^*=\frac{1}{d}1_d$, the $(i,j)^{th}$ entry is $\frac{1}{d}\delta_{ij}$. Thus up to a scaling $A_i$'s are trace orthonormal.

Now form $Z=\sum_{i=1}^d K_i\otimes A_i$. Since the $A_i$s are diagonal, $Z$ is a block matrix each of which is a diagonal matrix with the $E_{ij}\otimes E_{kk}$ entry being given by 
\[\sum_{s=1}^d (K_s)_{ij}(A_s)_{kk}=\sum_{s=1}^d (K_s)_{ij}v_{ks}=\sum_{s=1}^d (v_{ks}K_s)_{ij}.\]
This is the $(i,j)^{th}$ entry of $\sum_{s=1}^d v_{ks}K_s$. As $v_k\in NC(K)$, by definition, this matrix is an entrywise positive matrix. Hence the $(k,k)$ entry of $Z(i,j)$, the $(i,j)^{th}$ block of $Z$, is positive. Since $Z(i,j)$ is a diagonal matrix, all of whose diagonal entries are positive, it is positive semidefinite. So all the entries of $Z$ are positive semidefinite matrices. Hence $Z$ is entrywise positive. 
\end{proof}

\subsection*{CP/CPSD cones and PF maps}
At this juncture we introduce a few notions of symmetric matrices. A symmetric $n\times n$ matrix $X$ is called \textit{completely positive} (CP) if there exist nonnegative vectors $\{p_i\}_{i=1}^n\in \R_{+}^k$, for some $k\geq 1$, such that $X=(X_{i,j})=(\langle p_i,p_j\rangle)$, for all $1\leq i,j\leq n$. The set of $n\times n$ completely positive matrices, denoted by $\mathcal{C}\mathcal{P}^n$,  forms a pointed, full-dimensional closed convex cone which has been studied extensively in the literature (see \cite{cp-book} and references therein).
Next, a symmetric $n\times n$ matrix $X$ is said to be \textit{completely positive semidefinite} (CPSD) if there exist positive semidefinite matrices $P_1,\cdots, P_n\in M_k(\C)$, for some $k\geq 1$, such that $X=(Tr(P_iP_j))$. The set of all such matrices, denoted by  $\mathcal{C}\mathcal{S}_{+}^n$, is a convex set. This cone has been introduced to establish linear conic formulations for various quantum graph parameters (\cite{Laurent1}, \cite{roberson}). If we denote $\mathcal{D}\mathcal{N}\mathcal{N}^n$ to be the set of all $n\times n$ positive semidefinite and entrywise nonnegative (doubly nonnegative), then it is known that
\[\mathcal{C}\mathcal{P}^n\subseteq \mathcal{C}\mathcal{S}_{+}^n\subseteq \mathcal{D}\mathcal{N}\mathcal{N}^n.\] 
It is known that $\mathcal{C}\mathcal{P}^n= \mathcal{D}\mathcal{N}\mathcal{N}^n$
 for $n\leq 4$ and strict inclusion holds for $n\geq 5$ (\cite{minc}, \cite{diananda}). Frankel and Weiner (\cite{weiner}) gave an example of a $5\times 5$ matrix which is doubly nonnegative but not CPSD and in \cite{fawzi} it was shown that there exists a $5\times 5$ matrix which is CPSD but not CP. 
It was shown in \cite{Laurent-2} that any matrix $X$ lying in the closure of $\mathcal{C}\mathcal{S}_{+}^n$ admits a gram representation by positive elements $A_1,\cdots, A_n$ in some tracial von-Neumann algebra $(\mathcal{A},\tau)$. That is $X=(\tau(A_iA_j))$. 

From the proof of Lemma \ref{lemma:real kraus} it is evident that a quantum channel $\Phi\in\PF$ iff the Choi matrix $J(\Phi)$ lies inside the closure of $\mathcal{C}\mathcal{S}_{+}^{n^2}$. In this subsection we characterize positively factorizable maps on $M_n(\C)$ whose Choi matrix lie inside the set $\mathcal{C}\mathcal{P}^{n^2}$. 

\begin{theorem}\label{Prop:nonneg Kraus}
For a channel $\Phi$ with Kraus operators $K=(K_1,\cdots, K_d)$, the following statements are equivalent
\begin{enumerate}
\item $\Phi\in \PF$ via an abelian algebra.
\item One can choose a set of Kraus operators $\{L_i\}$ for $\Phi$ such that every $L_i$ is nonnegative.
\item The Choi matrix of $\Phi$, $J(\Phi)$, is a CP matrix.
\end{enumerate}
\end{theorem}
\begin{proof}
\textbf{1 $\Longrightarrow$ 2.}
 Using  Theorem \ref{Thm:abelian ancilla} we know that if $\Phi$ is in $\PF$ via an abelian algebra, then there are vectors $\{v_i\}$ and probabilities $\{p_i\}$ such that $\sum_i p_iv_iv_i^*=\frac{1}{d}I_d$ and $\sum_{j}v_{ij}K_j$ is nonnegative. Now define 
$L_i=\sum_{j}v_{ij}K_j$ and we check for any $X$, 
\[\sum_i p_iL_iXL_i^*=\sum_{i,j,k}p_iv_{ij}\bar{v_{ik}}
K_jXK_k^*=\sum_{j,k}K_jXK_k^*(\sum_{i}p_iv_{ij}
\bar{v_{ik}}).\]
Now note that from the equation $\sum_i p_iv_iv_i^*=\frac{1}{d}I_d$, the second sum
is $\delta_{jk}$. So we get
\[\sum_i p_iL_iXL_i^*=\sum_{j=1}^d K_jXK_j^*=\Phi(X).\]
Hence $\{\sqrt{p_i}L_i\}$ is a set of nonnegative Kraus operators for $\Phi$.

\textbf{2 $\Longrightarrow$ 3.}
If we can choose a set of non negative Kraus operators $\{K_i\}$ for $\Phi$, then the Choi matrix satisfies the relation $J(\Phi)=\sum_i vec(K_i)vec(K_i)^*$. As $vec(K_i)$ is a vector with nonnegative entries, by definition $J(\Phi)$ is a CP matrix.  

\textbf{3 $\Longrightarrow$ 1.} If $J(\Phi)$ is a CP matrix, then from the relation $$J(\Phi)=\sum_i vec(K_i)vec(K_i)^*,$$ one can choose a set of nonnegative Kraus operators for $\Phi$. Then any choice of $d$ orthogonal projections $\{A_i\}_{i=1}^d$ on a Hilbert space, would result in an operator $Z=\sum_{i=1}^d K_i\otimes A_i$, whose $(i,j)^{th}$ block is a positive linear combination of $A_i$ with the entries of $K_i$. As these entries are all positive numbers, we get that $Z(i,j)$ is a positive operator for any $(i,j)$ which means $\Phi\in \PF$. The fact that $\Phi$ factors through an abelian algebra follows from the fact that the algebra $\mathcal{A}$ generated by $\{A_i\}$ is abelian as the $A_i$ are orthogonal.
\end{proof}

From the result above the following corollary is immediate. One can compare the result of this corollary with the Theorem 4.2 in \cite{PR} where the map associated to local/classical correlations turns out to be a mixed permutation map. 
\begin{corollary}
Any quantum channel on $M_n(\C)$ with the permutation matrices as Kraus operators is in $\PF$ and the factorizing algebra can be taken to be abelian.
\end{corollary}
Note that the local/classical correlations in synchronous and bisynchronous games arise from abelian 
C$^*$-algebras \cite{hmkps}.
 
\begin{example}\label{examp:onb}
 The cone $S = \mathrm{cone}\{v_1,\cdots, v_d\}$ where the $v_i$'s form an o.n. basis for $\C^d$ is self-dual; it is, up to a unitary transformation, just the positive orthant of vectors whose entries are nonnegative. If $S \subseteq NC(K)$ for some $\Phi$ with Kraus operators $K = (K_1,\cdots, K_d)$, then $\Phi$ is in $\PF$; indeed $\Phi$ is in $\PF$ by means of an abelian algebra, as $v_i$ are an orthonormal basis each vector of which is in $NC(K)$, and thus they satisfy $\frac{1}{d}\sum_{i=1}^d v_iv_i^* = \frac{1}{d}I_d$.
\end{example}

\begin{remark}
It is worth noting the similarities between the conditions $S\subseteq NC(K)$ for $S = S^*$, the necessary condition for $\Phi$ to be in $\PF$, and $\{v_i\} \subseteq NC(K)$ for $\sum_i p_i v_iv_i^* = I_d$, the sufficient condition for $\Phi$ to be in $\PF$ by means of an abelian algebra. 

Note that both conditions require that $NC(K)$ be full-dimensional: in the first case, if there is a subspace $V \subsetneq \C^d$ such that $NC(K) \subseteq V$, then $V^{\perp}= V^* \subseteq NC(K)^* \subseteq NC(K) \subseteq V$, and since $V$ is not the full space, we can find a non-zero vector $v \in V^{\perp} \subseteq V$, which must satisfy $\langle v,v\rangle = 0$, a contradiction.
In the second case, if $\sum_i p_i v_iv_i^* = \frac{1}{d}I_d$, then for any vector $x \in \C^d$, we have that 

$$x = I_d x = d\sum_i p_i \langle v_i,x\rangle v_i$$ and so $x \in \mathrm{span}\{v_i\}$ for any $x \in \C^d$. 

What's more, as we saw in the previous example, there is a family of self-dual cones $S$, those generated by an orthonormal basis, such that $S \subseteq NC(K)$ is a sufficient condition for $\Phi$ to be in $\PF$. Indeed, if $d = 2$, the two conditions coincide; from \cite{barker1976self} we have that every self-dual cone in two dimensions is a cone generated by an orthonormal basis. Note that the analogous result fails for $d\geq 3$. 
\end{remark} 
\subsection{Examples and non-examples}
Here we note down some examples and non-examples of these maps.
\begin{example}
Consider the Werner-Holevo channel $\Phi:M_3(\mathbb{C})\rightarrow M_3(\mathbb{C})$ defined by \[\Phi(X)=\frac{1}{2}(
Tr(X)1-X^t),\] where $X^t$ denotes the transpose of $X$. 
One can check that a set of Kraus operators for $\Phi$ are given by the following three matrices:
\[
K_1=\begin{bmatrix}
0 & 0 & 0\\
0 & 0 & \frac{1}{\sqrt{2}}\\
0 & \frac{-1}{\sqrt{2}} & 0
\end{bmatrix}, \quad
K_2=\begin{bmatrix}
0 & 0 & \frac{-1}{\sqrt{2}}\\
0 & 0 & 0\\
\frac{1}{\sqrt{2}} & 0 & 0
\end{bmatrix}, \quad
K_3=\begin{bmatrix}
0 & \frac{1}{\sqrt{2}} & 0\\
\frac{-1}{\sqrt{2}} & 0 & 0\\
0 & 0 & 0
\end{bmatrix}.
\]
Now it follows that $NC(K)$ contains no self-dual cone. In fact $NC(K)=\{0\}$. If not, then assume there is a vector $y=(y_1, y_2, y_3)\in NC(K)$, then from the condition $\sum_i y_iK_i\geq 0$, we obtain 
\[\frac{1}{\sqrt{2}}\begin{bmatrix}
0 & {y_3} & {-y_2}\\
{-y_3} & 0 & {y_1}\\
{y_2} & {-y_1} & 0
\end{bmatrix}
\geq 0\]
Clearly the above matrix can not be entrywise nonnegative for any $(y_1,y_2,y_3)$ unless we  have $y_i=0$, for all $i$ and hence $y=0$. So this map is not PF.
\end{example}
\begin{example}
Consider the completely depolarizing channel on $M_n(\C)$
\[\Omega(x)=Tr(X)\frac{1}{n}.\]
One representation of this map is with the standard matrix units $E_{i,j}$. Indeed, one checks that $\Omega(x)=\frac{1}{n}\sum_{i,j=1}^n E_{i,j}XE_{i,j}^*$.
As the Kraus operators are nonnegative, by the Proposition \ref{Prop:nonneg Kraus} this map is in $\PF$.
\end{example}
More examples concerning Schur maps are given later in the paper.
%Werner-Holevo non-example and also the 5 by 5 in the section with kraus rank 2. Example.-Completely depolarizing channel, mixed permutation channel (see the theorem in factorizable via abelian ancilla). Also mention that  starting with a stochastic matrix $A=(a_{i,j})$, the channel built with Kraus operators $K_{i,j}=\sqrt{a_{i,j}}E_{i,j}$ is always in $\PF$ as the Kraus operators are nonnegative matrices.

\section{Closed under compositions and convex combinations}
Here we show that the $\PF$ maps are closed under compositions.
\begin{theorem}
If $\Phi,\Psi$ are in $\PF$ maps through von Neumann algebras $\mathcal{A}$ and $\mathcal{B}$, then $\Psi\circ\Phi$ is in $\PF$ map through $\mathcal{A}\otimes \mathcal{B}$.
\end{theorem}
\begin{proof}
Following the definition, if $\Phi,\Psi$ are in $\PF$, then there exist finite von-Neumann algebras $\mathcal{A},\mathcal{B}$ with traces $\tau_\mathcal{A},\tau_\mathcal{B}$ with operators $Z,W$ in $M_n\otimes \mathcal{A}$ and $M_n\otimes \mathcal{B}$ (respectively) such that $\forall X\in M_n$ we have

\[\Phi(X)=id\otimes \tau_\mathcal{A}(Z(X\otimes 1_\mathcal{A})Z^*)\ \ \text{and} \ \Psi(X)=id\otimes \tau_\mathcal{B}(W(X\otimes 1_\mathcal{B})W^*).\]
Moreover, if $\Phi(X)=\sum_{i=1}^p K_iXK_i^*$ and $\Psi(X)=\sum_{j=1}^q S_jXS_j^*$ are the Kraus decompositions of these two maps, then there are operators $\{A_1,\cdots, A_p\}\in \mathcal{A}$ and $\{B_1,\cdots,B_q\}\in \mathcal{B}$ such that 
\[Z=\sum_{i=1}^p K_i\otimes A_i \ \text{and} \ W=\sum_{j=1}^q S_j\otimes B_j,\]
with the property that
\begin{equation}\label{delta-func}
\tau_\mathcal{A}(A_iA_j^*)=\delta_{ij} \ 1\leq i,j\leq p, \ \text{and} \ \tau_\mathcal{B}(B_lB_m^*)=\delta_{lm}, 1\leq l,m\leq q.
\end{equation}
Here $\delta_{ij}$ is the Kronecker delta function and also if $Z=(Z(i,j))$ and $W=(W(i,j))$, then $Z(i,j)\succeq 0$ as well as $W(i,j)\succeq 0$, for all $i,j$.

Now we will show that the composition $\Psi\circ\Phi$ factors through the von-Neumann algebra $\mathcal{A}\otimes \mathcal{B}$, which is still finite as both $\mathcal{A},\mathcal{B}$ are finite. To this end, note that 
\[\Psi\circ\Phi(X)=\sum_{l=1}^q \sum_{i=1}^p S_l K_i X K_i^* S_l^*, \ \forall X\in M_n.\] 

Define \[\tilde{Z}=\sum_{i=1}^p K_i\otimes A_i\otimes 1_\mathcal{B}=(\sum_{i=1}^p K_i\otimes A_i)\otimes 1_\mathcal{B}=(Z(i,j)\otimes 1_\mathcal{B})\] and similarly  \[\tilde{W}=\sum_{j=1}^q S_j\otimes 1_\mathcal{A}\otimes B_j=(1_\mathcal{A}\otimes W_{i,j}).\] 
Here we use the isomorphism between $M_n\otimes \mathcal{A}\otimes \mathcal{B}$ and $M_n\otimes \mathcal{B}\otimes \mathcal{A}$.
Now let ${Q}=\tilde{W}\tilde{Z}$. Clearly it is an operator in $M_n\otimes \mathcal{A}\otimes \mathcal{B}$ with the property that the $(i,j)$th entry of ${Q}$ is $Q_{i,j}=\sum_{k}Z(k,j)\otimes W(i,k)$ which is positive.

We compute for any $X\in M_n(\C)$ 
\begin{align*}
& Q(X\otimes 1_\mathcal{A}\otimes 1_\mathcal{B})Q^*\\
&=\tilde{W}(\sum_i K_i\otimes A_i\otimes 1_\mathcal{B})(X\otimes 1_\mathcal{A}\otimes 1_\mathcal{B})(\sum_j K_j^*\otimes A_j^*\otimes 1_\mathcal{B})\tilde{W^*}\\
&=\tilde{W}(\sum_{i,j}K_iXK_j^*\otimes A_iA_j^*\otimes 1_\mathcal{B})\tilde{W}^*\\
&=(\sum_{l}S_l\otimes 1_\mathcal{A}\otimes B_l)(\sum_{i,j}K_iXK_j^*\otimes A_iA_j^*\otimes 1_\mathcal{B})(\sum_m S_m^*\otimes 1_\mathcal{A}\otimes B_m^*)\\
&=\sum_{l,m,i,j}S_l K_i X K_j^* S_m^*\otimes A_iA_j^*\otimes B_lB_m^*. 
\end{align*}
Now we trace out the system $\mathcal{A}\otimes \mathcal{B}$ and get 
\begin{align*}
&id\otimes \tau_\mathcal{A}\otimes \tau_\mathcal{B}(W(X\otimes 1_\mathcal{A}\otimes 1_\mathcal{B})W^*)\\
&=\sum_{l,m,i,j}S_l K_i X K_j^* S_m^*.\tau_\mathcal{A}(A_iA_j^*)\tau_\mathcal{B}(B_lB_m^*).
\end{align*}
Using the Equation \ref{delta-func}, we get 
\[id\otimes \tau_\mathcal{A}\otimes \tau_\mathcal{B}(Q(X\otimes 1_\mathcal{A}\otimes 1_\mathcal{B})Q^*)=\sum_{l,i}S_l K_i X K_i^*S_l^*=\Psi\circ\Phi(x).\]
Hence the result.
\end{proof}
\begin{proposition}
The set of $\PF$ maps are closed under convex combinations.
\end{proposition}

\begin{proof}
Let $\Phi,\Psi$ be two quantum channel which are positively factorizable. Suppose $\Phi,\Psi$ are represented by sets of Kraus operators $\{K_i\}_{i=1}^p$ and  $\{S_i\}_{i=1}^q$ respectively. Let $(\mathcal{A},\ta)$ and $(\mathcal{B},\tb)$ be two tracial von-Neumann algebras through which $\Phi,\Psi$ factors respectively. We will show that any convex combination $\mathcal{E}=\lambda \Phi+(1-\lambda)\Psi$ for $\lambda\in (0,1)$, is positively factorizable by means of the algebra $\mathcal{C}=\mathcal{A}\oplus \mathcal{B}$ with trace $\tau_{\mathcal{C}}=\lambda\ta\oplus(1-\lambda)\tb$.  

To this end, let $\{A_	i\}_{i=1}^p$ and $\{B_i\}_{i=1}^q$ be two sets of operators in $\mathcal{A},\mathcal{B}$ respectively by which the two channels factorize. Let $\{C_i\}_{i=1}^{p+q}$ be given by
$C_i=\big[(\sqrt{\lambda}^{-1}A_i)\oplus 0\big]\in \mathcal{C}$ if $1\leq i\leq p$, and $C_i=\big[ 0\oplus(\sqrt{(1-\lambda)}^{-1} B_i)\big]\in \mathcal{C}$ if $p+1\leq i\leq p+q$. It follows that for $1\leq i\leq p$ and $p+1\leq j\leq p+q$ or vice versa, we have $C_iC_j^*=0$. For any other case,
\[\tau_{\mathcal{C}}(C_iC_j^*)=\lambda\ta(\lambda^{-1}A_iA_j^*)=
\delta_{i,j}\]
or 
\[\tau_{\mathcal{C}}(C_iC_j^*)=(1-\lambda)\tb((1-\lambda)^{-1}
B_{i-p}B_{j-p}^*)=
\delta_{i,j}\]
So the operators $\{C_i\}$ are trace orthonormal. Now note that one set of Kraus operators of $\mathcal{E}$ is given by $\{\sqrt{\lambda}K_i\}_{i=1}^p\cup
\{\sqrt{(1-\lambda)}S_j\}_{j=1}^q$.  It follows that the operator 
\[X=\sum_{i=1}^p \sqrt{\lambda}K_i\otimes C_i+\sum_{i=p+1}^{p+q}\sqrt{(1-\lambda)}S_i
\otimes C_i\]
is the required operator in $M_n\otimes \mathcal{C}$ through which $\mathcal{E}$ factors positively. Indeed, the only thing we need to check is that the entries of $X$ are all positive elements of $\mathcal{C}$. This follows from the fact that the entries of $\sum_{i=1}^p K_i\otimes A_i$ and $\sum_{i=1}^{q} S_i\otimes B_i$ are all positive elements of $\mathcal{A},\mathcal{B}$ respectively. And direct sum of positive elements are positive.
\end{proof}

\section{PF Schur product maps}
An interesting set of examples to consider are the Schur product channels, channels of the form $\Phi(X) = X\circ C$ for a correlation matrix, a PSD matrix $C$ with $1$s down the diagonal. 
In this section we analyse the necessary and sufficient conditions for a Schur map $S_C$, corresponding to a correlation $C=(c_{i,j})$ matrix, to be in $\PF$.
\begin{proposition}\label{prop:Schur-PF}
A Schur map $S_{C}:M_n(\C)\rightarrow M_n(\C)$ is in $\PF$ iff
there exist positive operators $Z_1,Z_2\cdots,Z_n$ in a finite von-Neumann algebra $(\mathcal{A},\tau)$ such that \[C=(c_{i,j})=(\tau(Z_iZ_j)).\]
\end{proposition}
\begin{proof}
Let $S_C \in \PF$. Then there is a von-Neumann algebra $(\mathcal{A},\tau)$ and element $Z=(Z_{i,j})$ such that $z_{i,j}$ are all positive and 
\[S_C(X)=id\otimes \tau(Z(X\otimes 1)Z^*).\]
Now applying this form on matrix units, we get
\[(c_{i,j}E_{i,j})=S_C(E_{i,j})=\sum_{k,l}
\tau(Z_{i,l}Z_{j,k})
E_{k,l}.\]
Now note that $S_C(E_{i,i})=E_{i,i}$. Hence from the above equation we get
\begin{equation}\label{eq-Schur}
E_{i,i}=S_C(E_{i,i})=\sum_{k,l}
\tau(Z_{i,l}Z_{i,k})E_{k,l}.
\end{equation}
Comparing coefficients, we get $k\neq l$, $\tau(Z_{i,l}Z_{i,k})=0$ and $k=l\neq i$ $\tau(Z_{i,k}Z_{i,k})=0$. By the faithfulness of $\tau$, we get $Z_{i,k}=0$ for all $k\neq i$, for all $i$. This means $Z$ is a block diagonal with positive elements $Z_{k,k}$ in the diagonal blocks, that is, $Z=\sum
_{k}E_{k,k}\otimes Z_{k,k}$.

Labeling $Z_k$ as $Z_{k,k}$
now it is clear from the Equation \ref{eq-Schur} that $(c_{i,j})=(\tau(Z_iZ_j))$. As the diagonal entries of $C$ are all 1, we must have $\tau(Z_i^2)=1$.

 Conversely, if there exist positive operators $Z_1,Z_2\cdots,Z_n$ in a finite von-Neumann algebra $(\mathcal{A},\tau)$ such that \[C=(c_{i,j})=(\tau(Z_iZ_j)),\] then define $Z=\sum_j E_{j,j}\otimes z_j$. 
 Then one verifies that 
 \[S_C(X)=id\otimes \tau(Z(X\otimes 1)Z^*).\]
\end{proof}

 If $C = \sum_{i=1}^d \lambda_i v_iv_i^*$, for o.n. eigenvectors $v_i$, then the Kraus operators for $S_C$ are $K_i = \lambda_i \mathrm{diag}(v_i)$. Define $w_i = \sum_{j=1}^d \lambda_j\overline{v_{ji}}e_j$; then $\langle w_i,w_j\rangle = c_{ij}$, and for any $c = (c_1,\cdots, c_d)$, 

$$K(c)_{jj} = \bigl(\sum_{i=1}^d c_i K_i\bigr)_{jj} = \langle w_j,c\rangle.$$

So $K(c) = \mathrm{diag}((\langle w_1,c\rangle, \cdots, \langle w_n,c\rangle))$ and this is positive if and only if $c \in \{w_1,\cdots, w_n\}^*$, the dual of the set of Gram vectors for $C$. That is, $NC(K) = \{w_1,\cdots, w_n\}^*$. Since if $\Phi$ is in $\PF$, then $NC(K) \supseteq W(A)$, we have that if a Schur product channel is positively factorizable , then

$$\{w_1,\cdots, w_n\}^* \subseteq W(A).$$

Thus, a necessary condition for a Schur product channel to be $\PF$ is that $\{w_1,\cdots, w_n\}^{*}$ contains a self-dual cone. 
\section{PF Maps with Choi rank 2}
We can generalize the relationship between Gram vectors and non-negative cones to general CP maps. And as a consequence we prove that for a quantum channel with Choi rank 2, the Choi matrix being nonnegative is a necessary and sufficient condition for the map to be positively factorizable.
\begin{proposition} Let $\Phi$ be a channel with Choi matrix $J(\Phi)$ and Kraus operators $\{K_i\}_{i=1}^d$. Let $\{w_{i,j}\}_{i,j=1}^{n,m}$ be Gram vectors for $J(\Phi)$; then $NC(K)^*$ is the cone generated by $\{w_{i,j}\}_{i,j=1}^{n,m}$. 
\end{proposition}
\begin{proof} If $k_i = \mathrm{vec}(K_i)$, then we have that $J(\Phi) = \sum_i k_ik_i^*$. Notice that $K(v) = \sum_{i=1}^d v_i K_i \geq 0$ if and only if $\sum_i v_i k_i \geq 0$. 

Also  note that $w_{i,j} =\sum_{q=1}^d k_{q,ij} e_q$ form a set of Gram vectors for $J(\Phi)$, since the $((i,j),(k,l))$ entry of $J(\Phi)$ is $$\sum_{q=1}^d \langle E_{ij} \otimes E_{kl},k_qk_q^*\rangle = \sum_{q=1}^d \overline{k_{q,ik}}k_{q,jl} = \langle w_{ik}, w_{jl}\rangle.$$

Hence, if $K = \sum_{q=1}^d k_q e_q^*$ is the matrix with $k_q$ as its columns, it has $w_{ij}^*$ as its rows; and so $\sum_i v_i k_i = K(v)\geq 0$ if and only if $\langle w_{ij}, v\rangle \geq 0$ for all $(i,j)$. 
\end{proof}

Thus, $NC(K)^*$ is always polyhedral cone.

\begin{theorem} Let $\Phi$ be a completely positive map with Choi-rank $2$, then $\Phi$ is in $\PF$ if and only if $J(\Phi)$ has all nonnegative entries.
\end{theorem}
\begin{proof}
Following \cite{barker1976self} we know that every two dimensional self-dual cone is isometric with the two dimensional orthant. So from Theorem \ref{thm-main} for $d=2$ and the example \ref{examp:onb},  $\Phi$ is in $\PF$ if and only if $NC(K)$ contains a cone generated by two orthogonal vectors $v_1,v_2 \in \R^2$. From the previous proposition, $NC(K)^*$ is the cone generated by $w_{i,j} = (k_{1,ij}, k_{2,ij})$, where $K_1, K_2$ are the two Kraus operators for $\Phi$. 

That $J(\Phi)$ has all nonnegative entries is a necessary condition is easy to see, so we only prove sufficiency. We do this by proving the contrapositive: if $\Phi$ is not $\PF$, $J(\Phi)$ has a negative entry. 

So, suppose $\Phi$ is not in $\PF$; then $NC(K)$ does not contain a cone generated by orthogonal $v_1,v_2 \subseteq \R^2$. As $NC(K)$ is a cone in $\R^2$, it must have two extremal rays, call them $u_1,u_2$, and the angle between them must be smaller than a right angle. Apply an orthogonal transformation to bring $u_1 \mapsto (1,0)$ and $u_2 \mapsto (a,b)$ where $(a,b) \geq 0$. Define $v_{ij}$ as the image of each $w_{ij}$ under the same transformation.

The cone generated by $v_{ij}$ must be the cone $S=\{(1,0), (a,b)\}^*$, the set of all vectors whose first component is positive, and that lies above the line $ax + by = 0$. We cannot have $a= 0$, $b> 0$, as then $NC(K)$ is an orthogonal transformation of the positive orthant; if $b = 0$, $NC(K)$ is simply a line. So first consider the case $a,b > 0$. 

In this case, the line $ax + by = 0$ is a downward sloping line through the origin, so the cone $S$ contains all of the positive orthant, plus a section of the orthant $\{(x,y): x>0, y<0\}$. As $v_{ij}$ must generate the same cone, there exist $(i,j)$, $(k,l)$ such that $v_{ij} = (0,y)$ with $y>0$ and $v_{kl} = (w,z)$ with $z<0$, and then $$\langle v_{ij}, v_{kl} \rangle = \langle w_{ij}, w_{kl}\rangle < 0$$ and so the $E_{ik}\otimes E_{jl}$ entry of $J(\Phi)$ is negative. 

If instead $b=0$, $NC(K)$ can be orthogonally transformed to the line segment $\{(a,0): a\geq 0\}$ in which case $S = NC(K)^* = \{(w,z): w \geq 0\}$. Once again, $v_{ij}$ must generate the same cone; this cone contains $(0,1)$ and $(0,-1)$, and so there must be $(i,j), (k,l)$ such that $v_{ij} = (0,w)$, $v_{kl} = (0,z)$ with $w > 0$ and $z < 0$, and once again 
$$\langle v_{ij}, v_{kl}\rangle = \langle w_{ij}, w_{kl}\rangle = J(\Phi)_{ik,jl} < 0.$$
\end{proof}

Thus, if the rank of $J(\Phi)$ is two, non-negativity of $J(\Phi)$ is a necessary and sufficient condition for $\Phi$ to be in $\PF$ 
\subsection{Maps with Choi rank greater than 2}
The assertion of the above theorem does not hold for maps with Choi rank greater than 2. Here we provide an example where the Choi matrix is nonnegative but the map is not positively factorizable. 

\noindent Consider the following 5 vectors in $\mathbb{R}^3$:

$v_0=\frac{1}{\sqrt{3}}(1,1,1), v_1=\frac{1}{\sqrt{2}}(0,1,1), v_2=\frac{1}{\sqrt{2}}(-1,0,1), v_3=\frac{1}{\sqrt{2}}(0,-1,1),v_4=\frac{1}{\sqrt{3}}(1,-1,1)$. 

Now consider the matrix $W=
[ \langle v_i,v_j\rangle]_{i,j=0}^4=\begin{bmatrix}
1 & \frac{2}{\sqrt{6}} & 0 & 0 & \frac{1}{3}\\
\frac{2}{\sqrt{6}} & 1 & \frac{1}{2} & 0 & 0\\
0 & \frac{1}{2} & 1 & \frac{1}{2} & 0\\
0 & 0 & \frac{1}{2} & 1 & \frac{2}{\sqrt{6}}\\
\frac{1}{3} & 0 & 0 & \frac{2}{\sqrt{6}} & 1
\end{bmatrix}.$ Then we have the following theorem:
\begin{theorem}\label{thm: no finite vn algebra}
There is no finite von-Neumann algebra $(\mathcal{A},\tau)$ with positive operators $A_0,\cdots,A_4$ in $\mathcal{A}$ such that 
\[W=(\tau(A_iA_j)).\]
\end{theorem}
\begin{proof}
First of all notice that $W_{i,j}=\langle v_i,v_j\rangle\geq 0$ for all $i,j$ and for any 
$i$,
\[\{v_i\}^{\perp}=\text{span}\{\{v_p,v_q\}: p\equiv i+2 \ (mod \ 5), q\equiv i+3 \ (mod \ 5)\}.\]
Now if there exists positive elements $A_0,\cdots, A_4$ in some finite von-Neumann algebra
$(\mathcal{A},\tau)$ with $W_{i,j}=\tau(A_iA_j)$, then certainly we will have for all $i$, $\tau(A_i^2)=1$ and $\tau(A_iA_p)=0=\tau(A_iA_q)$ for $p\equiv i+2 \ (mod \ 5), q\equiv i+3 \ (mod \ 5)$. Note that, as the elements $\{A_i\}$ are positive, it follows that $\tau(A_iA_p)=\tau(A_p^{1/2}A_iA_p^{1/2})=0$. Using the faithfulness of $\tau$, we get $A_p^{1/2}A_iA_p^{1/2}=0$. Which implies $A_p^{1/2}A_i^{1/2}=0$.
Then it follows that $$A_pA_i=A_p^{1/2}A_p^{1/2}A_{i}^{1/2}A_{i}^{1/2}=0.$$ Similarly, $A_iA_p=0$.
Hence we  have $\text{for} \ p\equiv i+2 \ (mod \ 5), q\equiv i+3 \ (mod \ 5),$
\begin{equation}\label{vN-eq1}
A_pA_i=A_iA_p=0=A_iA_q =A_qA_i
\end{equation}
Now since the assignment $v_i\rightarrow A_i$ preserves inner product, it follows that $\text{Span}\{v_0,\cdots,v_4\}=\text{Span}\{A_0,\cdots, A_4\}$ with the Euclidian structure arising from the trace $\tau$. As $\{v_0,v_2,v_3\}$ forms a basis in $\mathbb{R}^3$, we have $v_1\in \text{Span}\{v_0, v_2,v_3\}$. Hence there exists constants $a,b,c$ such that \[A_1=aA_0+bA_2+cA_3.\]   
Multiplying the above equation by $A_0$ from the left and using the orthogonality from the equation \ref{vN-eq1} we get 
\begin{equation}\label{vN-eq-2}
A_0A_1=aA_0^2.
\end{equation}
Similarly note that the set $\{v_1, v_3, v_4\}$ forms a basis in $\mathbb{R}^3$ and hence expressing $v_0$ in this basis we will find  
constants $\alpha,\beta,\gamma$ such that 
\[A_0=\alpha A_1+\beta A_3+\gamma A_4.\] 
Multiplying $A_1$ from the right and using the equation \ref{vN-eq1} again we obtain
\begin{equation}\label{vN-eq-3}
A_0A_1=\alpha A_1^2.
\end{equation} 
Hence from \ref{vN-eq-2} and \ref{vN-eq-3} we have $a A_0^2=\alpha A_1^2$. Taking trace we get $a=\alpha$ (since $\tau(A_i^2)=1, \ \forall i$). This means $A_0^2=A_1^2$ and consequently $A_0=A_1$. This is a contradiction as $\tau(A_0A_1)=\langle v_0,v_1\rangle=\frac{2}{\sqrt{6}}(\neq 1)$.

\end{proof}
Now we show an example of a map whose Choi matrix is nonnegative but it is not positively factorizable.
\begin{theorem}
Consider the correlation matrix \[W=\begin{bmatrix}
1 & \frac{2}{\sqrt{6}} & 0 & 0 & \frac{1}{3}\\
\frac{2}{\sqrt{6}} & 1 & \frac{1}{2} & 0 & 0\\
0 & \frac{1}{2} & 1 & \frac{1}{2} & 0\\
0 & 0 & \frac{1}{2} & 1 & \frac{2}{\sqrt{6}}\\
\frac{1}{3} & 0 & 0 & \frac{2}{\sqrt{6}} & 1
\end{bmatrix}.\] 
The Schur product map associated with $W$ is not in $\mathcal{P}\mathcal{F}(5)$, although the Choi matrix of this map is an entrywise nonnegative psd matrix.
\end{theorem}
\begin{proof}
The proof follows from the Proposition \ref{prop:Schur-PF} and the Theorem \ref{thm: no finite vn algebra}.
\end{proof}
%%%%%%%%%%%%%%%%%%%%%%%%%%%5
\begin{remark}
Note that the vectors $\{v_0,\cdots, v_4\}$ appearing above give rise to a self-dual polyhedral cone in $\mathbb{R}^3$ as was shown in \cite{barker1976self}. Matrices like $W$ above which can not be realized as trace inner product in  any finite von-Neumann algebra has been investigated before (see \cite{weiner},\cite{Laurent1}) in connection with the strict inclusion of completely positive semidefinite cone inside the nonnegative cone. 
Our example above is new and related to self-dual cones. It also provides an example of a doubly nonnegative matrix which is not in the  closure of CPSD matrices.
Whether there is a connection to self-dual cones and these correlation matrices that can not be realized as trace inner-product in von-Neumann algebras is an interesting avenue for future research.
\end{remark}
\subsection{Unextendible Product Bases}
The example above has an interesting connection to what are known as \emph{unextendible product bases}.
\begin{definition} [see \cite{divincenzo}] Given two sets of vectors $\{u_i\}_{i=1}^n \subseteq \C^{d_1}$, $\{v_i\}_{i=1}^n \subseteq \C^{d_2}$, the vectors form an unextendible product basis (UPB) if the tensor products

$$\{u_i\otimes v_i\}_{i=1}^n \subseteq \C^{d_1}\otimes \C^{d_2}$$ are mutually orthogonal, and if there exists no product vector $x\otimes y \in \C^{d_1}\otimes \C^{d_2}$ in the orthogonal complement $\{u_i\otimes v_i\}^{\perp}$. 
\end{definition}

In fact, the idea behind a UPB can be extended to arbitrarily many tensor factors, though we will continue to focus only on the bipartite case.

The orthogonality condition on the tensor products means that for each pair $(i,j)$, $i\neq j$, at least one of $\langle u_i,u_j\rangle$ or $\langle v_i,v_j\rangle$ must be $0$; the pattern of zero inner products is captured by means of the \emph{orthogonality graph} of the UPB, $G$. This is a graph $G$ with an edge $(i,j)$ if and only if $\langle u_i,u_j\rangle = 0$; then necessarily the orthogonality graph for the $\{v_i\}$ is the complement of $G$. 

We briefly collect some important observations on bipartite UPBs, for more see \cite{divincenzo} \cite{alon}.

\begin{lemma}\label{divinlem}\cite{divincenzo} Suppose $\{u_i\otimes v_i\}_{i=1}^n \subseteq \C^d\otimes \C^d$ is a bipartite UPB; then for any subset $S\subseteq \{1,2,\cdots, n\}$, $|S|=n-d+1$, the set $\{u_i\}_{i\in S}$ must span $\C^d$ (and by symmetry, the same is true for $\{v_i\}_{i\in S}$).
\end{lemma}
\begin{proof} This is a slightly weaker verson of the bipartite, $d_1 = d_2$ case of Lemma 1 from \cite{divincenzo}. The argument is simple enough that we will recapitulate it here. 

Consider any partition of $\{1,2,\cdots ,n \} = S \cup S^C$ where $|S|=n-d+1$. Since $|S^C| = d-1$, there is necessarily a vector $y\in \{v_i\}_{i\in S^C}^{\perp}$,  as $\{v_i\}_{i\in S^C}$ can span a subspace of $\C^d$ whose dimension is at most $d-1$. Thus $\langle x\otimes y,u_i\otimes v_i\rangle = 0$ for any $i\in S^C$. Since $\{u_i\otimes v_i\}$ form a UPB, there must be a $j\in S$ such that $\langle x\otimes y, u_j\otimes v_j\rangle \neq 0$ for each $x$. However, if $\{u_j\}_{j\in S}$ does not span $\C^d$, then it spans a space of dimension at most $d-1$, and so we can pick $x\in \{u_j\}_{j\in S}^{\perp}$ to contradict the above. 
\end{proof}

\begin{corollary} A bipartite UPB on $\C^d\otimes \C^d$ has minimal size $n=2d-1$.
\end{corollary}
This is a consequence of Lemma $1$ from \cite{divincenzo}: if $n\leq 2(d-1)$ then pick a partition $\{1,2,\cdots,n\} = S\cup S^C$ where $|S|,|S^C|\leq d-1$. Then there is necessarily an $x\in \{u_i\}_{i\in S}^{\perp}$ and a $y\in \{v_j\}_{j\in S^C}^{\perp}$ and $\langle x\otimes y,u_i\otimes u_j \rangle = 0$ for all $i$. 

\begin{corollary} In the minimal case $n=2d-1$, Lemma \ref{divinlem} says that any subset of size $d$ of the $\{u_i\}$ must be a basis.
\end{corollary}

Alon and Lovasz were able to show that the condition that all subsets of a certain cardinality must span a subspace of some minimum dimension constrains the possible orthogonality graphs underyling a UPB. The following Lemma is their Theorem $3.1$.
\begin{lemma}\label{alonlemma}\cite{alon} The orthogonality graph of the vectors $\{u_i\}_{i=1}^n\subseteq\C^d$ that make up one tensor factor of a UPB must be $n-d$ vertex-connected.
\end{lemma}

From the foregoing, we see that the smallest possible $n$ for a UPB in $\C^3\otimes \C^3$ is $n=5$, and it can be shown that the only orthogonality graph for such a UPB is the $5$-cycle, $C_5$. We note that this is the orthogonality graph of the matrix $W$ from our Theorem $6.4$ above. The following shows that this is not entirely a coincidence.

In what follows, given a graph $G$, we use the notation $N(i)$ to refer to the neighbours of a vertex $i$, that is $N(i) = \{j: (i,j) \in E(G)\}$, and $N[i]$ is the closed neighbour set, inclusive of $i$: $N[i] = N(i)\cup \{i\}$.

\begin{lemma}\label{sikoralem} Let $G$ be the orthogonality graph of a set of unit vectors $\{u_i\}_{i=1}^n \subseteq \C^d$. If there exist two indices $(i,j)$ such that $\{u_i,u_j\}$ is a linearly independent set, and such that $\{u_k\}_{k\in N[i]}$ and $\{u_k\}_{k\in N[j]}$ are both of size $d$ and linearly independent, then the associated Gram matrix is not CPSD.
\end{lemma}
\begin{proof} This is a generalization of Theorem $34$ from \cite{sikora}, whose proof we adapt. First of all, notice that the two sets must each be a basis of the space $\C^d$, as they are both independent of cardinality $d$. Thus, 
\begin{align} u_i & = \sum_{k\in N[j]} c_k u_k \\
u_j & = \sum_{k\in N[i]} d_k u_k
\end{align} and so if the Gram matrix is CPSD for some PSD matrices $\{A_i\}$, the same equations must hold for $A_k$ replacing $u_k$. Then multiply the first equation by $A_j$ on the right, and $A_i$ by the second equation on the left to obtain

\begin{align} A_iA_j &= \sum_{k\in N[j]} c_k A_kA_j & = c_j A_j^2 \\
A_i A_j & = \sum_{k\in N[i]} d_k A_iA_k & = d_i A_i^2 
\end{align} 
where the last equality on each line follows from the fact that, by definition of the orthogonality graph, $\langle u_j,u_k\rangle = \mathrm{tr}(A_kA_j) = 0$ for all $k\in N[j]$, and so $A_kA_j = 0$ and similarly for $A_iA_k$ for $k\in N[i]$. 

So $A_i^2$ is a scalar multiple of $A_j^2$; since every positive operator has a unique positive square root, $A_i$ must be a scalar multiple of $A_j$ as well, contradicting that $\{u_i,u_j\}$ is independent.
\end{proof}

\begin{theorem} Let $\{u_i\}_{i=1}^{2d-1}\subseteq \C^d$ be part of a UPB on $\C^d\otimes \C^d$ of minimal size $2d-1$. The Gram matrix of these vectors cannot be CPSD.
\end{theorem}
\begin{proof} From Lemma \ref{alonlemma} the orthogonality graph of $G$ has to have minimum degree at least $n-d = d-1$, so that $\{u_k\}_{k\in N[i]}$ is a set of size $d$ for every $i$, and hence from Lemma \ref{divinlem}, must span; thus for any two indices, the sets $\{u_k\}_{k\in N[i]}$ and $\{u_k\}_{k\in N[j]}$ both span $\C^d$; so long as any two indices $u_i$, $u_j$ are linearly independent, the conditions of Lemma \ref{sikoralem} apply. 

Since there is at least one edge in $G$, there is at least one pair $\{u_i,u_j\}$ that are mutually orthogonal and hence linearly independent. 
\end{proof}

This result in itself is not so interesting, as there is no guarantee that the Gram matrix resulting from a UPB is even DNN, in which case we would certainly not expect it to be CPSD. What is perhaps more interesting, is at least in the case $d=3$, the orthogonality graph and rank of a UPB are enough to rule out a Gram matrix being CPSD. 

\begin{theorem} No correlation matrix whose rank is $3$ and whose orthogonality graph is the orthogonality graph of a minimal UPB in $\C^3 \otimes \C^3$ can be CPSD.
\end{theorem}
\begin{proof} This essentially follows from the observation that any rank-$3$ correlation matrix whose orthogonality graph is isomorphic to the $5$-cycle, the only possible orthogonality graph for a $5$-member UPB on $\C^3\otimes \C^3$, necessarily has all of its $3\times 3$ minors corresponding to a neighbour set of some vertex positive; that is, all size-$3$ subsets of the Gram vectors of the form $\{u_k\}_{k\in N[i]}$ form a basis. Thus the proof from the previous Theorem may be applied. 

To see this, suppose $i$ is a vertex whose neighbours are $j,k$ and suppose there is a linear dependency among the three vectors. By the fact that they are neighbours in the orthogonality graph, $u_j,u_k \in u_i^{\perp}$ and so each pair $\{u_i,u_j\}$ and $\{u_i,u_k\}$ is linearly independent--thus it must be that $\{u_j,u_k\}$ are linearly dependent, i.e, $u_j = \lambda u_k$. But if so, then any vector orthogonal to $u_j$ must be orthogonal to $u_k$ and vice-versa, and so the two vectors have the same neighbour set; but no vertices in the $5$-cycle have the same neighbour set. 

\end{proof}

Since a correlation matrix $C$ is CPSD if and only if the associated Schur product map $T_C(X) = X\circ C$ is PF, we see that there appears to be an interesting connection between the combinatorial structure of the zeros in the Choi matrix of a map, and whether it can be PF; in particular, if the zero pattern of the Choi matrix is related to the orthogonality graph of a minimal UPB, it may prevent the map from being PF.  
\section{acknowledgements}
The authors thank Vern Pauslen for reading the earlier version of this paper and his comments. The authors also like to thank the referee for carefully reading the manuscript.
MR is supported by the Start-up Research Grant by the Science and Engineering Research Board (Govt. of India). 

\bibliographystyle{alpha}

\bibliography{PFbib}

\newcommand{\etalchar}[1]{$^{#1}$}
\begin{thebibliography}{AMR{\etalchar{+}}19}

\bibitem[AD06]{A-D}
Claire Anantharaman-Delaroche.
\newblock On ergodic theorems for free group actions on noncommutative spaces.
\newblock {\em Probability Theory and Related Fields}, 135(4):520, 2006.

\bibitem[AL01]{alon}
Noga Alon and L{\'a}szl{\'o} Lov{\'a}sz.
\newblock Unextendible product bases.
\newblock {\em Journal of combinatorial theory. Series A}, 95(1):169--179,
  2001.

\bibitem[AMR{\etalchar{+}}19]{graph-iso}
Albert Atserias, Laura Man\v{c}inska, David~E. Roberson, Robert \v{S}\'{a}mal,
  Simone Severini, and Antonios Varvitsiotis.
\newblock Quantum and non-signalling graph isomorphisms.
\newblock {\em J. Combin. Theory Ser. B}, 136:289--328, 2019.

\bibitem[BDM{\etalchar{+}}99]{divincenzo}
Charles~H Bennett, David~P DiVincenzo, Tal Mor, Peter~W Shor, John~A Smolin,
  and Barbara~M Terhal.
\newblock Unextendible product bases and bound entanglement.
\newblock {\em Physical Review Letters}, 82(26):5385, 1999.

\bibitem[BF76]{barker1976self}
George~Philip Barker and James Foran.
\newblock Self-dual cones in euclidean spaces.
\newblock {\em Linear Algebra and its Applications}, 13(1-2):147--155, 1976.

\bibitem[BLP15]{Laurent-2}
Sabine Burgdorf, Monique Laurent, and Teresa Piovesan.
\newblock On the closure of the completely positive semidefinite cone and
  linear approximations to quantum colorings.
\newblock In {\em 10th {C}onference on the {T}heory of {Q}uantum {C}omputation,
  {C}ommunication and {C}ryptography}, volume~44 of {\em LIPIcs. Leibniz Int.
  Proc. Inform.}, pages 127--146. Schloss Dagstuhl. Leibniz-Zent. Inform.,
  Wadern, 2015.

\bibitem[BSM03]{cp-book}
Abraham Berman and Naomi Shaked-Monderer.
\newblock {\em Completely positive matrices}.
\newblock World Scientific Publishing Co., Inc., River Edge, NJ, 2003.

\bibitem[Dia62]{diananda}
P.~H. Diananda.
\newblock On non-negative forms in real variables some or all of which are
  non-negative.
\newblock {\em Proc. Cambridge Philos. Soc.}, 58:17--25, 1962.

\bibitem[FGP{\etalchar{+}}15]{fawzi}
Hamza Fawzi, Jo\~{a}o Gouveia, Pablo~A. Parrilo, Richard~Z. Robinson, and
  Rekha~R. Thomas.
\newblock Positive semidefinite rank.
\newblock {\em Math. Program.}, 153(1, Ser. B):133--177, 2015.

\bibitem[FNT17]{spectrahedra}
Tobias Fritz, Tim Netzer, and Andreas Thom.
\newblock Spectrahedral containment and operator systems with
  finite-dimensional realization.
\newblock {\em SIAM J. Appl. Algebra Geom.}, 1(1):556--574, 2017.

\bibitem[FW14]{weiner}
P\'{e}ter~E. Frenkel and Mih\'{a}ly Weiner.
\newblock On vector configurations that can be realized in the cone of positive
  matrices.
\newblock {\em Linear Algebra Appl.}, 459:465--474, 2014.

\bibitem[HM11]{H-M1}
Uffe Haagerup and Magdalena Musat.
\newblock Factorization and dilation problems for completely positive maps on
  von {N}eumann algebras.
\newblock {\em Comm. Math. Phys.}, 303(2):555--594, 2011.

\bibitem[HM15]{H-M2}
Uffe Haagerup and Magdalena Musat.
\newblock An asymptotic property of factorizable completely positive maps and
  the {C}onnes embedding problem.
\newblock {\em Comm. Math. Phys.}, 338(2):721--752, 2015.

\bibitem[HMPS19]{hmkps}
J.~William Helton, Kyle~P. Meyer, Vern~I. Paulsen, and Matthew Satriano.
\newblock Algebras, synchronous games, and chromatic numbers of graphs.
\newblock {\em New York J. Math.}, 25:328--361, 2019.

\bibitem[LP15]{Laurent1}
Monique Laurent and Teresa Piovesan.
\newblock Conic approach to quantum graph parameters using linear optimization
  over the completely positive semidefinite cone.
\newblock {\em SIAM J. Optim.}, 25(4):2461--2493, 2015.

\bibitem[MM63]{minc}
John~E. Maxfield and Henryk Minc.
\newblock On the matrix equation {$X^{\prime} X=A$}.
\newblock {\em Proc. Edinburgh Math. Soc. (2)}, 13:125--129, 1962/63.

\bibitem[MRv{\etalchar{+}}17]{relaxation}
Laura Man\v{c}inska, David~E. Roberson, Robert \v{S}\'{a}mal, Simone Severini,
  and Antonios Varvitsiotis.
\newblock Relaxations of graph isomorphism.
\newblock In {\em 44th {I}nternational {C}olloquium on {A}utomata, {L}anguages,
  and {P}rogramming}, volume~80 of {\em LIPIcs. Leibniz Int. Proc. Inform.},
  pages Art. No. 76, 14. Schloss Dagstuhl. Leibniz-Zent. Inform., Wadern, 2017.

\bibitem[OP16]{ortiz-paulsen}
Carlos~M. Ortiz and Vern~I. Paulsen.
\newblock Quantum graph homomorphisms via operator systems.
\newblock {\em Linear Algebra Appl.}, 497:23--43, 2016.

\bibitem[PR21]{PR}
Vern~I. Paulsen and Mizanur Rahaman.
\newblock Bisynchronous games and factorizable maps.
\newblock {\em Ann. Henri Poincar\'{e}}, 22(2):593--614, 2021.

\bibitem[PSS{\etalchar{+}}16]{psstw}
Vern~I. Paulsen, Simone Severini, Daniel Stahlke, Ivan~G. Todorov, and Andreas
  Winter.
\newblock Estimating quantum chromatic numbers.
\newblock {\em J. Funct. Anal.}, 270(6):2188--2222, 2016.

\bibitem[PSVW18]{sikora}
Anupam Prakash, Jamie Sikora, Antonios Varvitsiotis, and Zhaohui Wei.
\newblock Completely positive semidefinite rank.
\newblock {\em Mathematical Programming}, 171(1):397--431, 2018.

\bibitem[Rob16]{roberson}
David~E. Roberson.
\newblock Conic formulations of graph homomorphisms.
\newblock {\em J. Algebraic Combin.}, 43(4):877--913, 2016.

\end{thebibliography}

\end{document}